\documentclass{amsart}
\usepackage[utf8]{inputenc}

\usepackage{amsmath,amsthm,amssymb}
\usepackage[margin=1in]{geometry}

\usepackage{mathabx}
\usepackage{hyperref}
\usepackage[usenames, dvipsnames]{xcolor}

\usepackage{mathtools}
\mathtoolsset{showonlyrefs}

\definecolor{darkblue}{rgb}{0.0,0.0,0.3}
\hypersetup{colorlinks,breaklinks,
  linkcolor=darkblue,urlcolor=darkblue,
anchorcolor=darkblue,citecolor=darkblue}

\theoremstyle{plain}
\newtheorem{theorem}{Theorem}[section]
\newtheorem*{theorem*}{Theorem}
\newtheorem{lemma}[theorem]{Lemma}
\newtheorem{proposition}[theorem]{Proposition}
\newtheorem*{proposition*}{Proposition}
\newtheorem{corollary}[theorem]{Corollary}
\newtheorem*{corollary*}{Corollary}

\theoremstyle{definition}
\newtheorem{remark}[theorem]{Remark}
\newtheorem{definition}[theorem]{Definition}

\numberwithin{equation}{section}

\newcommand{\Mod}[1]{\ (\mathrm{mod}\ #1)}
\DeclareMathOperator{\Disc}{Disc}
\DeclareMathOperator{\LDisc}{LDisc}
\DeclareMathOperator{\Gal}{Gal}
\DeclareMathOperator{\Ht}{ht}  

\title{Quantitative Hilbert irreducibility and almost prime values of polynomial discriminants}
\author[T.C. Anderson]{Theresa C. Anderson}
\author[A. Gafni]{Ayla Gafni}
\author[R.J. Lemke Oliver]{Robert J. Lemke Oliver}
\author[D. Lowry-Duda]{David Lowry-Duda}
\author[G. Shakan]{George Shakan}
\author[R. Zhang]{Ruixiang Zhang}
\date{\today}

\begin{document}

\begin{abstract}%
  We study two polynomial counting questions in arithmetic statistics via a
  combination of Fourier analytic and arithmetic methods. First, we obtain new
  quantitative forms of Hilbert's Irreducibility Theorem for degree $n$
  polynomials $f$ with
  $\mathrm{Gal}(f) \subseteq A_n$. We study this both for
  monic polynomials and non-monic polynomials. Second, we study lower bounds on
  the number of degree $n$ monic polynomials with almost prime discriminants,
  as well as the closely related problem of lower bounds on the number of
  degree $n$ number fields with almost prime discriminants.
\end{abstract}

\maketitle

\section{Introduction}

The study of statistics for objects of algebraic interest has been a source of
rich and deep advances in mathematics. One of the goals of a recent AIM
workshop was to make progress on counting problems by further incorporating
Fourier analytic techniques into arithmetic statistics. This project grew out of
that workshop and is an effort in that direction.

In this paper, we use analytic and arithmetic methods in tandem to study a
variety of arithmetic statistics related to polynomial counts. We hope to see
further and more refined applications of similar ideas in other counting
problems in the future.
We focus on two primary applications, both involving counting polynomials of
certain types. First, we study a quantitative version of \emph{Hilbert's
Irreducibility Theorem} (HIT). A precise statement follows below, but our
version gives upper bounds for the number of degree $n$ polynomials whose
Galois groups are subgroups of $A_n$. Our techniques apply equally well to
monic and non-monic polynomials, so we examine both.

To state our version of HIT, we need a few definitions. We let
$V_n(\mathbb{Z})$ denote the set of degree $n$ polynomials over $\mathbb{Z}$,
and let $V_n^{\mathrm{mon}}(\mathbb{Z}) \subset V_n(\mathbb{Z})$ denote the set
of monic degree $n$ polynomials. We also define
\begin{equation}
    V_n(H)
    =
    \{
      f(x) \in \mathbb{Z}[x]
      :
      f(x) = a_n x^n+a_{n-1}x^{n-1}+ \cdots + a_0,
      a_n \neq 0,
      \Ht(f) \leq H
    \},
\end{equation}
where we define the \emph{height} of a polynomial $f$, $\Ht(f)$, as the
maximum of the absolute value of the coefficients. Let $V_n^{\mathrm{mon}}(H) \subset V_n(H)$ denote
the subset of monic polynomials. Finally, let
\begin{equation}
    E_n^{\mathrm{mon}}(H) = \lvert \{ f(x) \in V_n^{\mathrm{mon}}(H), \Gal(f) \neq S_n\} \rvert.
\end{equation}

Van der Waerden~\cite{vdw1936} gave the first explicit bound for $E_n^\mathrm{mon}(H)$ and
conjectured that $E_n^{\mathrm{mon}}(H) \ll_n \lvert V_n^{\mathrm{mon}}(H) \rvert / H$.
Gallagher~\cite{Gallagher1973} used
the large sieve to improve van der Waerden's bound to
$
  E_n^{\mathrm{mon}}(H) \ll_n H^{n-1/2}(\log H)^{1-\gamma_n},
$
where $\gamma_n$ is a sequence of positive numbers satisfying $\gamma_n \sim (2\pi n)^{-1/2}$.
Zywina~\cite{zywina2010} further improved this by removing the power of $\log H$, and the record is work of Dietmann~\cite{Dietmann2013} who shows that
\begin{equation}\label{eqn:dietmann-an}
  E_n^{\mathrm{mon}}(H) \ll_n H^{n-2+\sqrt{2}+\epsilon},
\end{equation}
and of Chow and Dietmann \cite{ChowDietmann2020}, who solve van der Waerden's conjecture when $n \leq 4$.

Significantly stronger bounds are known for the number $E_n^{\mathrm{mon}}(H)^\prime$ of monic polynomials
with Galois group isomorphic to neither $S_n$ nor $A_n$.
There have been a number of recent results on this problem, including work of Zywina~\cite{zywina2010} and Dietmann~\cite{Dietmann2012}.  The record on this problem is the \emph{very} recent work of Chow and Dietmann~\cite{ChowDietmann2021}, who show using the determinant method that
\begin{equation}\label{eqn:chow-dietmann}
    E_n^\mathrm{mon}(H)^\prime \ll_n H^{n-1.017}, \quad n \not\in \{7,8,10\},
\end{equation}
which resolves van der Waerden's conjecture in these degrees, apart from bounding the number of polynomials whose Galois group is $A_n$.  Their methods also essentially apply when $n \in \{7,8,10\}$, but save a power of $H$ smaller than $1$.

Here, we improve the bounds on polynomials $f$ where $\Gal(f) \subseteq A_n$.
For $H \geq 1$, define
    \[
      E_n(H; A_n) := \lvert \{ f \in V_n(H) : \Gal(f) \subseteq A_n \} \rvert.
    \]
We note that $E_n(H; A_n)$ counts a distinguished subset of polynomials counted
by $E_n(H)$. Define $E_n^{\mathrm{mon}}(H; A_n)$ by restricting to monic
polynomials. We use a combination of arithmetic and analytic techniques to
prove the following.\@

\begin{theorem}\label{thm:hit-intro}
    Let $n \geq 3$ be an integer and let $H \geq 2$. Then for any $\epsilon > 0$,
    \begin{equation*}
      E_n(H; A_n) \ll_{n,\epsilon} H^{n+\frac{1}{3}+\frac{8}{9n+21} + \epsilon}
    \end{equation*}
    and
    \begin{equation}\label{eq:thm_monic_compare}
      E_n^{\mathrm{mon}}(H; A_n) \ll_{n,\epsilon} H^{n-\frac{2}{3}+\frac{2}{3n+3} + \epsilon}.
    \end{equation}
\end{theorem}

For $n \geq 8$, the bound \eqref{eq:thm_monic_compare} improves on Dietmann's
bound \eqref{eqn:dietmann-an} for the number of monic polynomials whose Galois
group is contained in $A_n$.  Combined with the work of Chow and Dietmann, this
improves the overall estimate on the error term in HIT to $E_n^\mathrm{mon}(H)
\ll_{n,\epsilon} H^{n - \frac{2}{3} + \frac{2}{3n+3} + \epsilon}$.  We also
note that Bhargava has announced a proof of van der Waerden's conjecture using
different methods.

Our approach to Theorem \ref{thm:hit-intro} is inspired by Gallagher's sieve theoretic approach (which also underlies Zywina's work), but instead of using the large sieve, we introduce a modification to the classical Selberg sieve in \S\ref{sec:modified_sieve}.  This modification allows us to connect the local conditions appearing in the sieve more properly to the M\"obius function over finite fields, provided we count the relevant polynomials with certain arithmetic weights, and is the key novelty in Theorem \ref{thm:hit-intro}.  We estimate the local density of the modified conditions by means of Poisson summation, with work of Porritt~\cite{porritt2018note} on bounds for the the Fourier transform of the M\"obius function appearing to control the error.

Our second application concerns lower bounds on the number of degree $n$
polynomials with \emph{almost prime} discriminants, i.e.\ discriminants with
relatively few distinct prime factors.  We draw inspiration from the following result of Taniguchi and
Thorne~\cite{TT20Levels}, who were in turn inspired by the folklore conjecture that there should be infinitely many fields of prime discriminant in every degree; this is known only for quadratic extensions, however.

\begin{theorem*}[\cite{TT20Levels}]
    There is an absolute constant $C_3>0$ such that for each $X >2$, there exist
    at least $C_3X/\log X$ cubic fields whose discriminant is squarefree, bounded above by $X$, and has
    at most $3$ prime factors, and there is an absolute constant $C_4>0$ such that for each $X >2$, there
    exist at least $C_4X/\log X$ quartic fields whose discriminant is squarefree, bounded above by $X$,
    and has at most $8$ prime factors.
\end{theorem*}
The cubic case improved an earlier result of Belabas and Fouvry~\cite{BelabasFouvry1999} which had
$3$ prime factors replaced by $7$.

In \S\ref{sec:almost_prime}, we first study the number of polynomials whose discriminants are almost
prime. We prove an almost prime discriminant result for all $n \geq 3$ that obtains discriminants with fewer prime factors than \cite{TT20Levels} if $n=4$.

\begin{theorem}
    Let $n \geq 3$, and let $H \geq 2$.  For any $r \geq 2n-3$, we have
        \[
            \#\{ f \in V_n^\mathrm{mon}(H) : \omega(\Disc(f)) \leq r\}
                \gg_{n,r} \frac{H^n}{\log H},
        \]
    where $\omega(\Disc(f))$ denotes the number of distinct primes dividing the
    discriminant of the polynomial $f$.
\end{theorem}

As the discriminant of a number field cut out by an irreducible polynomial divides that of the
polynomial, we can use lower bounds for counts of almost prime polynomial discriminants to get lower
bounds for almost prime number field discriminants. To make this comparison effective, we use
results from~\cite{LOTPolys} that bound the number of different polynomials of a given height that
cut out the same number field. This allows us to prove the following theorem.  We state a more precise version as Theorem \ref{thm:almost-prime-fields}.

\begin{theorem}\label{thm:almost-prime-intro}
   Let $n \geq 3$, and let $X \geq 2$.  There is a constant $\delta_n >0$ depending only on $n$, such that for any $r \geq 2n-3$, we have
        \[
            \#\{ F/\mathbb{Q} : [F:\mathbb{Q}]=n, \Disc(F) \leq X, \omega(\Disc(F)) \leq r\}
                \gg_{n} X^{\frac{1}{2} + \delta_n},
        \]
    where $\omega(\Disc(F))$ denotes the number of distinct primes dividing the
    discriminant of the field $F$.
\end{theorem}

In the quartic case $n=4$, Theorem \ref{thm:almost-prime-intro} improves on the quality of the
almost primes produced by Taniguchi and Thorne (achieving $r=5$ as opposed to $r=8$), but at the
expense of obtaining a worse lower bound on the number of such fields.  In fact, the lower bounds
obtained by Taniguchi and Thorne are of the expected order of magnitude for the number of prime
discriminant fields, which is $\asymp_n X / \log X$ for every $n$, while Theorem
\ref{thm:almost-prime-intro} falls short.  The reason for this is that to prove their theorem,
Taniguchi and Thorne use group actions on \emph{prehomogeneous vector spaces}. They are then able to
count certain lattice points related to the desired field counts, utilizing deep parametrization
theorems and Poisson summation.  Their method is powerful, but as it relies on parametrizations via
prehomogeneous vector spaces, it is only currently available for degrees less than or equal to $5$.
It is interesting to note that they are sometimes able to explicitly compute all Fourier
transforms~\cite{TT20Orbital}, but can prove their results using rougher estimates.

To get a result for all $n \geq 3$, we use a different approach that involves studying an
underlying Fourier transform directly and the \emph{almost prime sieve}. Our analysis centers on the
Fourier transform of the squarefree indicator function. For small degrees, it may be possible to
include additional arithmetic ingredients to improve our results.

Finally, to reach a wide audience, we have erred on the side of writing more details and
explanations. We hope for this to be an engaging, understandable paper.

\section*{Acknowledgements}

This work was supported by the National Science Foundation [DMS-1954407
to TCA; DMS-1856541 to RZ; and DMS-1926686 to RZ];
the Simons Collaboration in Arithmetic Geometry, Number Theory, and Computation
[546235 to DLD];
the Deutsche Forschungsgemeinschaft [(EXC-2047/1 - 390685813) to AG]
while AG was in residence at the Hausdorff Institute in Bonn, Germany;
Ben Green's Simons Investigator Grant [376201 to GS];
and the Ky Fan and Yu-Fen Fan Endowment Fund [to RZ] at the Institute for
Advanced Study.

This collaboration arose out of an AIM workshop on Fourier analysis, arithmetic statistics, and
discrete restriction organized by the first author, Frank Thorne, and Trevor Wooley.  The authors
thank AIM for the supportive working environment.  They would also like to thank Will Sawin for
helping strengthen this paper, and Manjul Bhargava, Kevin Hughes, Hong Wang, and Jiuya Wang for
useful conversations.

\section{Polynomials over finite fields}

We begin in this section by collecting some basic facts from algebraic number
theory on the reduction modulo primes of integer polynomials.  (See for
example~\cite[\S4.16]{jacobson2012basic} as a reference).

\begin{lemma}\label{lem:poly-disc}
    Let $f(x) \in \mathbb{Z}[x]$ be a polynomial, and let $\Disc(f) \in
    \mathbb{Z}$ denote its polynomial discriminant.  Then $\Disc(f) = 0$ if and
    only if $f(x)$ has a repeated factor (which happens over $\mathbb{C}$ if
    and only if it happens over $\mathbb{Z}$).  Moreover if $p$ is a prime
    number not dividing the leading coefficient of $f$, then $p \mid \Disc(f)$
    if and only if $f(x) \Mod p$ has repeated factors.
\end{lemma}

Notice that if $f(x) \in \mathbb{Z}[x]$ is irreducible, then it has no repeated
factors (since it has only one!).  It follows that $\Disc(f)$ must be
non-zero, and thus can be divisible by only finitely many primes.  In
particular, it will have repeated factors (mod $p$), or be of smaller degree,
only for those finitely many primes.  For all of the others, we have the
following connection between factorization types and Galois groups.

\begin{lemma}\label{lem:kronecker}
    Suppose $f(x) \in \mathbb{Z}[x]$ is irreducible with degree $n$.  Let $G
    \subseteq S_n$ be its Galois group, thought of as permuting the roots of
    $f(x)$.  Suppose $p$ is a prime not dividing the leading coefficient of
    $f(x)$ for which $f(x) \Mod p$ has no repeated factor.  Write
        \[
            f(x) = f_1(x) \dots f_r(x) \pmod{p},
        \]
    where each $f_i(x)$ is irreducible (mod $p$).

    Then there is an element of $G$ with cycle type $(\deg f_1)(\deg f_2)
    \cdots (\deg f_r)$.  In fact, this is true for any of the Frobenius
    elements associated to $p$.
\end{lemma}

In particular, if $\Gal(f) \subseteq A_n$, then the reduction of $f$ at
any prime subject to Lemma~\ref{lem:kronecker} must correspond to an even cycle
type.  As we are approaching our main theorem via sieves, it is the
complementary notion that is of most interest to us:

\begin{definition}\label{Odddefn}
    We say a polynomial $f(x) \in \mathbb{F}_p[x]$ is \emph{odd} if it has no
    repeated roots and the permutations with cycle type corresponding to the
    factorization type of $f(x)$ are odd.  Equivalently, $f(x)$ is odd if it
    has no repeated factors and the number of its irreducible factors with even
    degree is odd.
\end{definition}

\begin{lemma}\label{lem:mobius}
    A degree $n$ polynomial $f$ over $\mathbb{F}_p$ is odd precisely if $\mu_p(f) =
    (-1)^{n+1}$, where $\mu_p(f)$ is the M\"obius function over
    $\mathbb{F}_p[x]$.
\end{lemma}

\begin{proof}
    Suppose a squarefree polynomial $f$ of degree $n$ over $\mathbb{F}_p$ has
    factorization type $\lambda_1 \cdots \lambda_r$.  Let $N_{\text{odd}} =
    \#\{ i : \lambda_i \text{ odd}\}$ and $N_{\text{even}} = \#\{ i : \lambda_i
    \text{ even}\}$ count the number of odd and even $\lambda_i$.
    Then $f$ is odd if $N_{\text{even}}$ is odd, i.e.\ if
        \[
            (-1)^{N_{\text{even}}} = -1.
        \]
    However, notice that $N_{\text{even}} = r - N_{\text{odd}}$ and that
    $N_{\text{odd}} \equiv n \Mod 2$. Thus
        \[
            (-1)^{N_{\text{even}}}
                = (-1)^{r - n} = \mu(f) (-1)^n.
        \]
    The result follows.
\end{proof}

Since we are primarily interested in the reduction of integer polynomials $f$,
when the leading coefficient of $f$ is not $\pm 1$ the degree of the reduction
of $f$ may be smaller than that of $f$.  Consequently, for a polynomial
$f \in \mathbb{F}_p[x]$, we define
    \begin{equation} \label{eqn:mu-p-n-definition}
        \mu_{p,n}(f)
            = \begin{cases}
                \mu_p(f), & \text{if $\deg(f) = n$}, \\
                0, & \text{if $\deg(f) \ne n$},
            \end{cases}
    \end{equation}
Given an integer polynomial $f \in \mathbb{Z}[x]$, we define $\mu_{p,n}(f)$ in
the expected manner by means of the reduction of $f \Mod p$.  It follows from
the above discussions that $\mu_{p,n}(f)=0$ if and only if  $p$ divides the
product of the leading coefficient of $f$ with discriminant of $f$.  Consequently,
we define the quantity $\LDisc(f)$ to be this product.

To end this section, we summarize the above discussion in the following lemma.
\begin{lemma}\label{lem:Anneverodd}
    Let $f \in \mathbb{Z}[x]$ be a polynomial of degree $n>0$ with
    $\Gal(f) \subseteq A_n$. Then $\mu_{p, n} (f) \neq (-1)^{n+1}$ for
    every prime $p$.
\end{lemma}

\section{Fourier transforms of polynomials over finite fields}\label{FTsec}

Given a squarefree integer $d$, let $V_n(\mathbb{Z}/d\mathbb{Z})$ denote the
vector space of polynomials over $\mathbb{Z}/d\mathbb{Z}$ with degree at most
$n$, and let $V_n^{\mathrm{mon}}(\mathbb{Z}/d\mathbb{Z})$ denote the subset of
those that are monic of degree equal to $n$.
We identify the dual of $V_n(\mathbb{Z}/d\mathbb{Z})$ with
$(\mathbb{Z}/d\mathbb{Z})^{n+1}$ and the dual of
$V_n^\mathrm{mon}(\mathbb{Z}/d\mathbb{Z})$ with $(\mathbb{Z}/d\mathbb{Z})^{n}$.
We define the pairing between $V_n(\mathbb{Z}/d\mathbb{Z})$ and
$(\mathbb{Z}/d\mathbb{Z})^{n+1}$ coefficient-wise; namely, if
$f(x) = \sum_{i=0}^n a_i x^i$ and
$\mathbf{u} =(u_0,\dots,u_n) \in (\mathbb{Z}/d\mathbb{Z})^{n+1}$, we define
    \[
        \langle f, \mathbf{u} \rangle
            := \sum_{i=0}^n a_i u_i.
    \]
We define the pairing $\langle \cdot,\cdot\rangle_\mathrm{mon}$ between
$V_n^\mathrm{mon}(\mathbb{Z}/d\mathbb{Z})$ and $(\mathbb{Z}/d\mathbb{Z})^n$
analogously. We will typically omit ``$\mathrm{mon}$'' from the notation if
it's clear that we are working with monic polynomials.
If $\psi \colon V_n(\mathbb{Z}/d\mathbb{Z}) \to \mathbb{C}$ is a function, we
define its Fourier transforms
    \[
        \widehat\psi(\mathbf{u})
            := \frac{1}{d^{n+1}} \sum_{f \in V_n(\mathbb{Z}/d\mathbb{Z})}
               \psi(f) e_d(\langle f,\mathbf{u}\rangle),
               \quad e_d(x) := e^{\frac{2\pi i x}{d}}
    \]
for $\mathbf{u} \in (\mathbb{Z}/d\mathbb{Z})^{n+1}$ and
    \[
        \widehat\psi^\mathrm{mon}(\mathbf{v})
            := \frac{1}{d^{n}} \sum_{f \in V_n^\mathrm{mon}(\mathbb{Z}/d\mathbb{Z})}
               \psi(f) e_d(\langle f,\mathbf{u}\rangle)
    \]
for $\mathbf{v} \in (\mathbb{Z}/d\mathbb{Z})^{n}$.
Exploiting the natural map $V_n(\mathbb{Z}/d\mathbb{Z}) \to \prod_{p \mid d}
V_n(\mathbb{F}_p)$, we will be primarily interested in functions of the form
$\psi_d := \prod_{p \mid d} \psi_p$, where $\psi_p \colon V_n(\mathbb{F}_p) \to
\mathbb{C}$.  For such functions, the Fourier transform has a corresponding
factorization.

\begin{lemma}\label{lem:fourier-multiplicative}
    Let $d$ be a squarefree integer.  For each prime $p \mid d$, let
    $\psi_p\colon V_n(\mathbb{F}_p) \to \mathbb{C}$, and for any
    $f \in V_n(\mathbb{Z}/d\mathbb{Z})$,
    define $\psi_d(f) := \prod_{p\mid d} \psi_p(f)$.
    There are units
    $\alpha_p \in \mathbb{F}_p^\times$ such that for any
    $\mathbf{u} \in \mathbb{Z}^{n+1}$ and any $\mathbf{v} \in \mathbb{Z}^n$,
        \[
            \widehat \psi_d (\mathbf{u})
                = \prod_{p \mid d} \widehat\psi_p( \alpha_p \mathbf{u})
            \quad \text{and} \quad
            \widehat \psi^\mathrm{mon}_d (\mathbf{v})
                = \prod_{p \mid d} \widehat\psi^\mathrm{mon}_p( \alpha_p \mathbf{v}).
        \]
\end{lemma}
\begin{proof}
    This follows from the Chinese remainder theorem, and the proof is the same
    in the general and monic cases. We give the proof for the general case.
    It suffices to prove the lemma for a squarefree factorization $d=d_1d_2$.
    A polynomial $f \in V_n(\mathbb{Z}/d\mathbb{Z})$ projects to
    $f_1 \in V_n(\mathbb{Z}/d_1\mathbb{Z})$
    and $f_2 \in V_n(\mathbb{Z}/d_2 \mathbb{Z})$.
    Conversely, given two such polynomials $f_1,f_2$, there is a unique
    polynomial $f \in V_n(\mathbb{Z}/d\mathbb{Z})$ congruent to each, namely
        \[
            f = f_1 d_2 \overline{d_2} + f_2 d_1 \overline{d_1},
        \]
    where $\overline{d_2}$ is any choice of the multiplicative inverse of
    $d_2 \Mod d_1$, with $\overline{d_1}$ defined analogously.  Then
        \begin{align*}
            \widehat \psi_d(\mathbf{u})
                &= \frac{1}{d^{n+1}}
                   \sum_{f_1 \in V_n(\mathbb{Z}/d_1\mathbb{Z})}
                   \sum_{f_2 \in V_n(\mathbb{Z}/d_2\mathbb{Z})}
                   \psi_{d_1}(f_1) \psi_{d_2}(f_2)
                   e_d(\langle f_1 d_2 \overline{d_2} + f_2 d_1 \overline{d_1},\mathbf{u}\rangle)
                   \\
                &= \widehat \psi_{d_1}(\overline{d_2} \mathbf{u})
                   \widehat \psi_{d_2}(\overline{d_1} \mathbf{u}).
        \end{align*}
    The lemma follows.
\end{proof}

In subsequent sections, Fourier transforms of this type will naturally appear
after an application of Poisson summation on the integer lattices
$\mathbb{Z}^{n+1}$ and $\mathbb{Z}^n$.  The next two lemmas will be used to
control the Fourier side of this application.

\begin{lemma}\label{lem:sharp-fourier-sum}
    Let $d$ be squarefree and suppose $\psi_d(f) = \prod_{p \mid d} \psi_p(f)$
    is a function where each $\psi_p$ satisfies
    $\widehat{\psi}_p(\mathbf{u}) \ll p^{-\alpha}$ for some $0 < \alpha < n$
    and every $\mathbf{u} \not \equiv \mathbf{0} \Mod p$, and furthermore
    $\widehat{\psi}_p(\mathbf{0}) \ll 1$. Then for any $X \geq 1$,
        \[
            \sum_{\substack{\mathbf{u} \in \mathbb{Z}^{n+1} \setminus \mathbf{0}
                            \\ \lvert u_i \rvert \leq X \, \forall i}}
            \lvert \widehat \psi_d(\mathbf{u}) \rvert
            \ll X^{n+1} d^{-\alpha}.
        \]
    Here, the sum is over $\mathbf{u} = (u_0, u_1, \ldots, u_n) \in
    \mathbb{Z}^{n+1} \setminus 0$ where each coordinate satisfies $\lvert u_i
    \rvert \leq X$. Each vector $\mathbf{u}$ is regarded in
    $(\mathbb{Z}/d\mathbb{Z})^{n+1}$ via the projection map.

    Similarly, if
    for each prime $p \mid d$, $\widehat \psi^\mathrm{mon}_p(\mathbf{0}) \ll 1$ and $\widehat \psi^\mathrm{mon}_p(\mathbf{v}) \ll p^{-\beta}$
    for some $0 < \beta < n-1$ and every $\mathbf{v} \not \equiv \mathbf{0}
    \Mod p$, then
        \[
            \sum_{\substack{\mathbf{v} \in \mathbb{Z}^{n} \setminus \mathbf{0}
                            \\ |v_i| \leq X \, \forall i}}
            \lvert \widehat \psi^\mathrm{mon}_d(\mathbf{v}) \rvert
            \ll X^{n} d^{-\beta}.
        \]
\end{lemma}
\begin{proof}
    We prove only the general case, the monic case following mutatis mutandis.
    There are fewer than $X^{n+1}$ choices of $\mathbf{u}$ such that
    $\mathbf{u} \not \equiv 0 \Mod p$ for each prime divisor $p$ of $d$. Thus
    the total contribution from these $\mathbf{u}$ is no larger than the
    asserted quantity, by Lemma~\ref{lem:fourier-multiplicative} and our
    assumption that $\widehat{\psi}_p(\mathbf{u}) \ll p^{-\alpha}$.

    It only remains to consider those $\mathbf{u}$ that are congruent to
    $\mathbf{0}$ modulo at least one prime divisor of $d$.
    For each divisor $m \neq 1$ of $d$, let $U_m$ denote the set of
    $\mathbf{u} \in \mathbb{Z}^{n+1} \setminus \mathbf{0}$ such that $m$ is the
    maximal divisor of $d$ with $\mathbf{u} \equiv 0 \Mod m$. Stated
    differently, to each $\mathbf{u}$ we associate the maximal $m \mid d$ such
    that $\mathbf{u} \equiv 0 \Mod m$ and thereby partition these $\mathbf{u}$
    into sets $U_m$.

    For each $\mathbf{u} \in U_m$, Lemma~\ref{lem:fourier-multiplicative} gives
    that
        \[
            \widehat \psi_d(\mathbf{u})
            =
            \widehat \psi_m(\mathbf{0})
            \widehat \psi_{d/m}(c \mathbf{u})
            \ll
            m^{\alpha} d^{-\alpha},
        \]
    where $c$ is some unit depending on $d$ and $m$. As $\# U_m \ll
    (X/m)^{n+1}$, it follows that
        \[
            \sum_{\substack{\mathbf{u} \in U_m
                            \\
                            \lvert u_i \rvert \leq X \, \forall i}}
            \lvert \widehat \psi_d(\mathbf{u}) \rvert
            \ll
            \left(\frac{m}{d}\right)^\alpha \left( \frac{X}{m} \right)^{n+1}
            \ll
            \frac{X^{n+1}}{d^\alpha m^{n+1-\alpha}}.
        \]
    As $d$ is squarefree, we have that
        \[
            \sum_{m \mid d} \frac{1}{m^{n+1-\alpha}}
            =
            \prod_{p \mid d} \left( 1 + \frac{1}{p^{n+1-\alpha}} \right),
        \]
    which is absolutely bounded since $\alpha < n$. This proves the claim.
\end{proof}

\begin{lemma}\label{lem:smooth-fourier-sum}
    Make the same assumptions as in Lemma~\ref{lem:sharp-fourier-sum}.  If $\phi \colon \mathbb{R}^{n+1} \to \mathbb{R}$ is Schwartz, then for any $X >0$,
        \[
            \sum_{\mathbf{u} \in \mathbb{Z}^{n+1}} \widehat \psi_d(\mathbf{u}) \phi(\mathbf{u}/X)
                = \widehat \psi_d(\mathbf{0}) \phi(\mathbf{0}) + O_{\phi, n}(X^{n+1} d^{-\alpha}),
        \]
    and if $\phi \colon \mathbb{R}^{n} \to \mathbb{R}$ is Schwartz, then
        \[
            \sum_{\mathbf{v} \in \mathbb{Z}^{n}} \widehat \psi^{\mathrm{mon}}_d(\mathbf{v}) \phi(\mathbf{v}/X)
                = \widehat \psi^\mathrm{mon}_d(\mathbf{0}) \phi(\mathbf{0}) + O_{\phi, n}(X^{n} d^{-\beta})
        \]
    for any $X > 0$.
\end{lemma}
\begin{proof}
  We briefly describe the general case, as the monic case is nearly identical.
  The idea is to use the rapid decay of $\phi$ and a form of annular partial
  summation.
  As $\phi$ is Schwartz, we have that $\lvert \phi(\mathbf{u}) \rvert \ll_\phi \lvert
  \mathbf{u} \rvert^{-n-3}$, hence
    \begin{align*}
      \sum_{\mathbf{u} \in \mathbb{Z}^{n+1}} \widehat \psi_d(\mathbf{u}) \phi(\mathbf{u}/X)
      -
      \widehat \psi_d(\mathbf{0}) \phi(\mathbf{0})
      &=
      \sum_{\mathbf{u} \in \mathbb{Z}^{n} \setminus \mathbf{0}} \widehat \psi_d(\mathbf{u}) \phi(\mathbf{u}/X)
      \ll_\phi
      \sum_{k \geq 0}
      \sum_{\substack{%
        \mathbf{u} \in \mathbb{Z}^{n} \setminus \mathbf{0}
        \\
        kX \leq \lvert \mathbf{u} \rvert \leq (k+1)X
      }} \lvert \widehat \psi_d(\mathbf{u}) \rvert \lvert \phi(\mathbf{u}/X) \rvert
      \\
      &\ll_\phi
      X^{n+1} d^{-\alpha}
      +
      \sum_{k \geq 1}
      \Big(\sum_{\substack{%
        \mathbf{u} \in \mathbb{Z}^{n} \setminus \mathbf{0}
        \\
        kX \leq \lvert \mathbf{u} \rvert \leq (k+1)X
      }} \lvert \widehat \psi_d(\mathbf{u}) \rvert \Big)
      k^{-n-3}
      \\
      &\ll_\phi
      X^{n+1} d^{-\alpha}
      +
      \sum_{k \geq 1}
      \big((k+1) X\big)^{n+1} d^{-\alpha}
      k^{-n-3}
      \ll_{\phi,n}
      X^{n+1} d^{-\alpha},
    \end{align*}
    where we have repeatedly applied Lemma~\ref{lem:sharp-fourier-sum}.
\end{proof}

Finally, we consider the functions that will be of interest to us in the next section, recalling
relatively recent work of Porritt~\cite{porritt2018note} on the function field analogue of bounds
for sums $\max_{\theta} \lvert \sum_{n \leq x} \mu(n) e(n\theta) \rvert$.  There is also work of
Bienvenu and L\^{e}~\cite{BienvenuLe} that is qualitatively of the same quality as Porritt's, but
less precise for our particular purpose. Additionally, there is also work of Dietmann, Ostafe, and
Shparlinski \cite{DietmannOstafeShparlinski} that exploits cancellation in the Fourier transform of
the M\"obius function in a closely related sieve problem;
see in particular~\cite[Lemma 2.7, Lemma 3.4]{DietmannOstafeShparlinski}.

\begin{lemma}\label{lem:fourier-mobius}
    Let $n \geq 3$, $p$ be prime, and define
        \[
            \psi_p(f)
                := \frac{1 + (-1)^{n+1} \mu_{p,n}(f)}{2},
        \]
    where $\mu_{p,n}$ is as in~\eqref{eqn:mu-p-n-definition}.
    Then $\widehat \psi_p(\mathbf{0}) = \widehat \psi^\mathrm{mon}_p(\mathbf{0}) = 1/2$
    and $\widehat \psi_p(\mathbf{u}),\widehat \psi^\mathrm{mon}_p(\mathbf{u}) \ll_n p^{\frac{1-n}{4}}$
    for $\mathbf{u} \not \equiv \mathbf{0} \Mod p$.
\end{lemma}
\begin{proof}
    The claim about $\widehat \psi_p(\mathbf{0})$ and $\widehat \psi_p^\mathrm{mon}(\mathbf{0})$
    follows from the classical fact that
        \[
            \sum_{f \in V_n(\mathbb{F}_p)} \mu_p(f)
            =
            \sum_{f \in V_n^\mathrm{mon}(\mathbb{F}_p)} \mu_p(f)
            =
            0
        \]
    for any $n \geq 2$.
    For $\mathbf{u} \not \equiv \mathbf{0}$, the claim about $\widehat \psi_p^\mathrm{mon}(\mathbf{u})$
    follows from~\cite[Theorem 1]{porritt2018note}.  For $\widehat \psi_p(\mathbf{u})$,
    we note that if $f(x) = a_n x^n + \cdots + a_0 \in \mathbb{F}_p[x]$ with
    $a_n \in \mathbb{F}_p^\times$, then $\mu_p(f) = \mu_p(f/a_n)$.  Consequently,
        \[
            \widehat \psi_p(\mathbf{u})
            =
            \frac{1}{p} \sum_{c \in \mathbb{F}_p^\times} \widehat \psi_p^\mathrm{mon}(c \mathbf{u}),
        \]
    which again may be bounded by~\cite[Theorem 1]{porritt2018note}.
\end{proof}

Combining Lemma~\ref{lem:smooth-fourier-sum} with Lemma~\ref{lem:fourier-mobius},
we immediately obtain the following corollary.

\begin{corollary}\label{cor:smooth-mobius}
    Let $n \geq 3$, let $\psi_p(f) = \frac{1 + (-1)^{n+1} \mu_{p,n}(f)}{2}$ for
    each prime $p$, and for squarefree $d$, let $\psi_d(f) = \prod_{p \mid d} \psi_p(f)$.
    If $\phi\colon \mathbb{R}^{n+1} \to \mathbb{R}$ is Schwartz, then for any
    $X > 0$ and any squarefree $d$,
        \[
            \sum_{\mathbf{u} \in \mathbb{Z}^{n+1}} \widehat \psi_d(\mathbf{u}) \phi(\mathbf{u}/X)
                = \frac{\phi(\mathbf{0})}{2^{\omega(d)}} + O_{\phi, n}(X^{n+1} d^{\frac{1-n}{4}}),
        \]
    where $\omega(d)$ denotes the number of distinct prime divisors of $d$.
    Similarly, if $\phi\colon\mathbb{R}^{n} \to \mathbb{R}$ is Schwartz, then
    for any $X > 0$ and any squarefree $d$,
        \[
            \sum_{\mathbf{v} \in \mathbb{Z}^{n}} \widehat \psi_d(\mathbf{v}) \phi(\mathbf{v}/X)
                = \frac{\phi(\mathbf{0})}{2^{\omega(d)}} + O_{\phi, n}(X^{n} d^{\frac{1-n}{4}}).
        \]
\end{corollary}

\section{A modified Selberg sieve}\label{sec:modified_sieve}

In this section, we introduce a modified version of the classical Selberg
sieve.  Our goal is to prove the following version, which we later specialize
using the results from the previous section.

\begin{proposition}\label{prop:modified-selberg}
    Let $n$ be a positive integer, $H$ and $D$ be real with $H, D \geq 1$,
    $\phi \colon V_n(\mathbb{R}) \rightarrow \mathbb{R}$ be a non-negative
    Schwartz function, and $\{\lambda_d\}$ be a sequence of real numbers
    indexed by squarefree integers $d \leq D$, with $\lambda_1 = 1$.
    Then
    \begin{equation}\label{eq:generalizedselberg}
        \sum_{\substack{f \in V_n(\mathbb{Z})
                        \\ \Gal(f) \subseteq A_n
                        \\ \LDisc(f) \ne 0}}
        \frac{\phi(f/H)}{2^{\omega(\LDisc(f))}}
        \leq
        \sum_{d_1,d_2} \lambda_{d_1} \lambda_{d_2}
        \sum_{f \in V_n(\mathbb{Z})} \phi(f/H) \prod_{p \mid [d_1,d_2]}
        \left( \frac{1 + (-1)^{n+1}\mu_{p,n}(f)}{2}\right).
    \end{equation}
\end{proposition}

This proposition can be viewed as a generalization of the Selberg sieve. Before
giving the proof, we first describe what can be obtained by the classical
Selberg sieve. (For a treatment of the classical Selberg sieve,
see~\cite[\S7]{friedlander2010opera}.) As in the statement of the proposition, we'll
assume $\lambda_d$ is a sequence of real numbers indexed by squarefree $d \leq D$ with
$\lambda_1 = 1$.

Using the classical Selberg sieve, we would start with the fact that
\begin{equation}\label{eq:classicalselberg}
    \sum_{f \in V_n(\mathbb{Z})} \phi(f/H)
    \left(
      \sum_{d: f \Mod{p} \, \text{is odd}, \forall p \mid d}
      \lambda_d
    \right)^2 \geq 0.
\end{equation}

On one hand, expanding the left hand side of~\eqref{eq:classicalselberg} we see it
is equal to
\begin{equation}\label{selbergRHS}
    \sum_{d_1, \, d_2} \lambda_{d_1} \lambda_{d_2}
    \sum_{f \in V_n(\mathbb{Z}): f \Mod{p} \, \text{is odd for every} \, p \mid [d_1,d_2]} \phi(f/H).
\end{equation}

On the other hand, when $f \in V_n (\mathbb{Z})$ and $\Gal(f) \subseteq A_n$, by
Lemma~\ref{lem:kronecker} we see that $f \Mod p$ is never odd for prime $p$ and thus
$\sum_{d: f \Mod{p} \, \text{is odd}, \forall p \mid d} \lambda_d = \lambda_1 = 1$.
By the non-negativity of $\phi$, we see (\ref{eq:classicalselberg}) is at least
\[
  \sum_{\substack{f \in V_n(\mathbb{Z}) \\ \Gal(f) \subseteq A_n}} \phi (f/H).
\]
Hence we have
\begin{equation}\label{eq:selbergsievebound}
    \sum_{\substack{f \in V_n(\mathbb{Z}) \\ \Gal(f) \subseteq A_n}} \phi (f/H)
    \leq \sum_{d_1,d_2} \lambda_{d_1} \lambda_{d_2}
    \sum_{f \in V_n(\mathbb{Z}): f \Mod p \, \text{is odd for every} \, p \mid [d_1,d_2]} \phi(f/H).
\end{equation}
The inequality~\eqref{eq:generalizedselberg} in Proposition~\ref{prop:modified-selberg}
should be compared with~\eqref{eq:selbergsievebound}.

We initially attempted to use~\eqref{eq:selbergsievebound} instead
of~\eqref{eq:generalizedselberg}, but the results are less satisfactory. The main reason
is that the characteristic function $1_{p, n}^{\mathrm{odd}}$ of odd polynomials in
$\mathbb{F}_p [x]$ of degree $n$ have very large Fourier coefficients away from $0$.

This is due to the fact (following from Lemma~\ref{lem:mobius}) that
$1_{p, n}^{\mathrm{odd}} = \frac{(-1)^{n+1}\mu_{p, n} + \mu_{p, n}^2}{2}$,
since $\mu_{p, n}^2$ is supported on square-free polynomials of degree exactly $n$.
As noted in Section~\ref{FTsec}, we expect the Fourier transform of $\mu_{p,
n}$ to behave well (i.e.\ be small) away from $0$, but one can show that
$\mu_{p, n}^2$ has large Fourier coefficients away from $0$ (see Remark \ref{largecoefficients} for a similar phenomenon in the monic case).

In order to circumvent this issue, we modify the Selberg sieve to produce the key
inequality~\eqref{eq:generalizedselberg}. The right hand side
of~\eqref{eq:generalizedselberg} maintains the strong Fourier decay of $\mu_{p, n}(f)$ (as
shown in Lemma~\ref{lem:fourier-mobius} and Corollary~\ref{cor:smooth-mobius}) in the
local computations after Poisson summation.

Compared to the classical Selberg sieve, the right hand side
of~\eqref{eq:generalizedselberg} has more complicated local factors
$\frac{1 + (-1)^{n+1}\mu_{p,n}(f)}{2}$ that can take the value $1/2$ in addition to the
typical $1$ and $0$. On the left hand side, we have a mild divisor-bound-type loss
$2^{-\omega(\LDisc(f))}$. This factor does not meaningfully detract from this application.

We now prove Proposition~\ref{prop:modified-selberg}.

\begin{proof}
    The fundamental idea of this proof is to use certain non-negative definite quadratic
    forms instead of the complete square $(\sum \lambda_d)^2$.

    For each $f \in V_n (\mathbb{Z})$, define the quadratic form $Q_f$ in the
    variables $\{\lambda_d\}$
        \[
            Q_f (\{\lambda_d\})
            =
            \sum_{d_1, d_2} \prod_{p \mid [d_1,d_2]}
            \left(\frac{1 + (-1)^{n+1}\mu_{p,n}(f)}{2}\right)
            \lambda_{d_1}\lambda_{d_2}.
        \]

    We claim that each $Q_f$ is non-negative definite. To see this, temporarily extend $Q_f$
    to a form on more variables
    $\{\lambda_d: d \text{ squarefree}, \text{ every prime factor of } d \text{ is }\leq D\}$
    using the same definition above. By definition the $(d_1, d_2)$-entry $q_{f, d_1, d_2}$
    of the matrix of $Q_f$ is equal to
    $\prod_{p \mid [d_1,d_2]} \left( \frac{1 + (-1)^{n+1}\mu_{p,n}(f)}{2}\right)$. In other
    words, if we write $\psi_{p} (f) = \frac{1 + (-1)^{n+1}\mu_{p,n}(f)}{2}$, then
        \[
            q_{f, d_1, d_2}
            =
            \bigg(\prod_{p \leq D, p \nmid d_1, p\nmid d_2} 1\bigg)
            \cdot
            \bigg(\prod_{p \leq D, p \nmid d_1, p\mid d_2} \psi_{p} (f)\bigg)
            \cdot
            \bigg(\prod_{p \leq D, p \mid d_1, p\nmid d_2} \psi_{p} (f)\bigg)
            \cdot
            \bigg(\prod_{p \leq D, p \mid d_1, p\mid d_2} \psi_{p} (f)\bigg).
        \]
    Hence the matrix of the (extended) form $Q_f$ is a tensor
    product of matrices $M_p (p \leq D \text{ prime})$ with
    $M_p = \begin{pmatrix} 1 & \psi_{p} (f)\\ \psi_{p} (f) & \psi_{p} (f) \end{pmatrix}$.
    More explicitly,
    $M_p = \begin{pmatrix} 1 & 1\\ 1 & 1 \end{pmatrix}$ for $\mu_{p, n} (f) = (-1)^{n+1}$,
    $M_p = \begin{pmatrix} 1 & 0 \\ 0  & 0 \end{pmatrix}$ for $\mu_{p, n} (f) = (-1)^{n}$ and
    $M_p = \begin{pmatrix} 1 & 1/2 \\ 1/2  & 1/2 \end{pmatrix}$ for $\mu_{p, n} (f) = 0$.
    From this we see the (extended) form $Q_f$ is non-negative definite. Since the original
    $Q_f$ is obtained by specifying all $\lambda_d = 0$ for $d > D$ in the extended form,
    the original form is also non-negative definite.

    We now show that whenever $\Gal(f) \subseteq A_n$ and $\LDisc(f) \neq 0$, we have
    $Q_f \geq 2^{-\omega(\LDisc(f))} \lambda_1^2 = 2^{-\omega(\LDisc(f))}$. It suffices to
    show this for the extended form $Q_f$ as described just above.
    When $\Gal(f) \subseteq A_n$ and $\LDisc(f) \neq 0$, Lemma~\ref{lem:Anneverodd} gives
    that $\mu_{p, n} (f) \neq (-1)^{n+1}$.
    Hence $M_p = \begin{pmatrix} 1 & 0 \\ 0  & 0 \end{pmatrix}$ for $p\nmid \LDisc(f)$ and
    $M_p = \begin{pmatrix} 1 & 1/2 \\ 1/2  & 1/2 \end{pmatrix}$ for $p \mid \LDisc(f)$.
    Note that as matrices $\begin{pmatrix} 1 & 1/2 \\ 1/2  & 1/2 \end{pmatrix} \geq
    \begin{pmatrix} 1/2 & 0 \\ 0  & 0 \end{pmatrix} \geq 0$, where $A \geq B$ means that
    $A-B$ is non-negative definite. Hence as a tensor product, the matrix of the (extended)
    form $Q_f$ is $\geq \begin{pmatrix} 2^{-\omega(\LDisc(f))} & 0 & \cdots & 0 \\ 0 & 0 &
    \cdots & 0 \\ \cdots & \cdots & \cdots & \cdots \\ 0 & 0 & \cdots & 0 \end{pmatrix}$,
    which shows that
    \begin{equation}\label{eq:Qf_lowerbound}
        Q_f \geq 2^{-\omega(\LDisc(f))} \lambda_1^2.
    \end{equation}

    The remainder of the proposition is now straightforward.
    The right hand side of~\eqref{eq:generalizedselberg} is equal to
    $\sum_{f \in V_n (\mathbb{Z})} \phi (f/H) Q_f$. On the other hand, applying the lower
    bound~\eqref{eq:Qf_lowerbound} gives precisely the left hand side
    of~\eqref{eq:generalizedselberg}.
\end{proof}

A similar proof gives also the monic version, which we record as the following
proposition.

\begin{proposition}\label{prop:modified-selberg-monic}
    Let $n$ be a positive integer, $H$ and $D$ be real with $H, D \geq 1$,
    $\phi \colon V^{\mathrm{mon}}_n(\mathbb{R}) \rightarrow \mathbb{R}$ be non-negative,
    and $\{\lambda_d\}$ be a sequence of real numbers indexed by squarefree
    integers $d \leq D$, with $\lambda_1 = 1$.  Then
    \begin{equation}\label{eq:generalizedselberg-monic}
        \sum_{\substack{f \in V^{\mathrm{mon}}_n(\mathbb{Z})
                        \\ \Gal(f) \subseteq A_n
                        \\ \Disc(f) \ne 0}}
        \frac{\phi(f/H)}{2^{\omega(\Disc(f))}}
        \leq
        \sum_{d_1,d_2} \lambda_{d_1} \lambda_{d_2}
        \sum_{f \in V^{\mathrm{mon}}_n(\mathbb{Z})} \phi(f/H) \prod_{p \mid [d_1,d_2]}
        \left( \frac{1 + (-1)^{n+1}\mu_{p}(f)}{2}\right).
    \end{equation}
\end{proposition}

\section*{Proof of Theorem~\ref{thm:hit-intro}}

In the proof, we will use a slightly atypical form of Poisson summation. We first state this and
give its proof.

\begin{lemma}\label{lem:poisson_summation}
  Fix $d \geq 2$.
  Let $\phi \colon \mathbb{R}^n \longrightarrow \mathbb{C}$ be Schwartz, and let $\psi_d \colon
  (\mathbb{Z} / d \mathbb{Z})^n \longrightarrow \mathbb{C}$ be any function.
  Let $\widehat{\phi} \colon \mathbb{R}^n \longrightarrow \mathbb{C}$ and $\widehat{\psi}_d \colon
  (\mathbb{Z} / d \mathbb{Z})^n \longrightarrow \mathbb{C}$ denote the Fourier transforms
  \begin{equation*}
    \widehat{\phi}(\mathbf{u})
    =
    \int_{\mathbb{R}^n} e(\langle \mathbf{x}, \mathbf{u} \rangle) \phi(\mathbf{x}) d\mathbf{x}
    \qquad \text{and} \qquad
    \widehat{\psi}_d(\mathbf{u})
    =
    \frac{1}{d^n} \sum_{\mathbf{x} \in (\mathbb{Z} / d \mathbb{Z})^n}
    e_d(\langle \mathbf{x}, \mathbf{u} \rangle) \psi_d(\mathbf{x}),
  \end{equation*}
  where $e(x) = e^{- 2 \pi i x}$ and $e_d(x) = e^{2 \pi i x / d}$. Then
  \begin{equation*}
    \sum_{\mathbf{x} \in \mathbb{Z}^n} \phi(\mathbf{x}) \psi_d(\mathbf{x})
    =
    \sum_{\mathbf{u} \in \mathbb{Z}^n}
    \widehat{\phi}\big( \frac{\mathbf{u}}{d} \big) \widehat{\psi}_d(\mathbf{u}).
  \end{equation*}
\end{lemma}

\begin{proof}
  Each $\mathbf{x} \in \mathbb{Z}^n$ can be written uniquely as $\mathbf{x} = d \mathbf{y} +
  \mathbf{z}$ where $\mathbf{y} \in \mathbb{Z}^n$ and $\mathbf{z} = (z_1, \ldots, z_n)$ satisfies $0
  \leq z_i \leq d - 1$. Then using the fact that $\psi_d$ is $(\mathbb{Z} / d \mathbb{Z})^n$-periodic,
  \begin{align*}
    \sum_{\mathbf{x} \in \mathbb{Z}^n} \phi(\mathbf{x}) \psi_d(\mathbf{x})
    &=
    \sum_{\substack{\mathbf{z} \in \mathbb{Z}^n \\ 0 \leq z_i \leq d-1}}
    \Big(%
      \sum_{\mathbf{y} \in \mathbb{Z}^n}
      \phi(d\mathbf{y} + \mathbf{z})
    \Big)
    \psi_d(\mathbf{z})
    =
    \sum_{\substack{\mathbf{z} \in \mathbb{Z}^n \\ 0 \leq z_i \leq d-1}}
    \frac{1}{d^n} \sum_{\mathbf{u} \in \mathbb{Z}^n}
    e_d(\langle \mathbf{z}, \mathbf{u} \rangle) \widehat{\phi}\big( \frac{\mathbf{u}}{d} \big)
    \psi_d(\mathbf{z})
    \\
    &=
    \sum_{\mathbf{u} \in \mathbb{Z}^n}
    \widehat{\phi}\big( \frac{\mathbf{u}}{d} \big)
    \Big(%
      \frac{1}{d^n}
      \sum_{\substack{\mathbf{z} \in \mathbb{Z}^n \\ 0 \leq z_i \leq d-1}}
      e_d(\langle \mathbf{z}, \mathbf{u} \rangle)
      \psi_d(\mathbf{z})
    \Big)
    =
    \sum_{\mathbf{u} \in \mathbb{Z}^n}
    \widehat{\phi}\big( \frac{\mathbf{u}}{d} \big)
    \widehat{\psi}_d(\mathbf{u}).
  \end{align*}
  The second equality uses classical Poisson summation on $\mathbb{Z}^n$, and the last equality
  follows from the definition of $\widehat{\psi}_d$ and embedding $\mathbf{z}$ into
  $(\mathbb{Z}/d\mathbb{Z})^n$.
\end{proof}

Next, we use Lemma~\ref{lem:poisson_summation} and Proposition~\ref{prop:modified-selberg} to prove
Theorem~\ref{thm:hit-intro}.

\begin{theorem}
    Let $n \geq 3$ be an integer and let $H \geq 2$ be real. Define $V_n(\mathbb{Z}; H)$ to be
    the set of polynomials $f = \sum a_i x^i$ in $V_n(\mathbb{Z})$ with $\max \, \lvert a_i
    \rvert \leq H$. Define $V_n^{\mathrm{mon}}(\mathbb{Z}; H)$ similarly. Then
        \[
            \sum_{\substack{f \in V_n(\mathbb{Z};H)
                            \\ \Gal(f) \subseteq A_n
                            \\ \LDisc(f) \ne 0}}
            \frac{1}{2^{\omega(\LDisc(f))}}
            \ll_n
            H^{n+\frac{1}{3}+\frac{8}{9n+21}} (\log H)^{\frac{4}{3n+7}}
        \]
    and
        \[
            \sum_{\substack{f \in V_n^\mathrm{mon}(\mathbb{Z};H)
                            \\ \Gal(f) \subseteq A_n
                            \\ \Disc(f) \ne 0}}
            \frac{1}{2^{\omega(\Disc(f))}}
            \ll_n
            H^{n - \frac{2}{3} + \frac{2}{3n+3}} (\log H)^{\frac{4}{3n+3}}.
        \]
\end{theorem}

\begin{proof}
    Choose a Schwartz function $\phi \colon V_n(\mathbb{R}) \rightarrow \mathbb{R}$ that is
    greater than or equal to $1$ on polynomials whose coefficients lie in $[-1, 1]$. For
    $f \in V_n(\mathbb{Z})$, we let
        \[
            \psi_d(f) = \prod_{p \mid d} \left( \frac{1 + (-1)^{n+1} \mu_{p,n}(f)}{2}\right).
        \]
    We apply Proposition~\ref{prop:modified-selberg}.
    The sum over $f$ on the right-hand side of~\eqref{eq:generalizedselberg} can be
    written as
        \[
            \sum_{f \in V_n(\mathbb{Z})} \phi(f/H) \psi_{[d_1,d_2]}(f).
        \]
    To apply Poisson summation as in Lemma~\ref{lem:poisson_summation}, we
    identify $V_n(\mathbb{Z})$ with $\mathbb{Z}^{n+1}$, write $f = \mathbf{x} \in
    \mathbb{Z}^{n+1}$, and define $\Phi(\mathbf{x}) := \phi(\mathbf{x}/H)$. Applying
    Lemma~\ref{lem:poisson_summation} to $\sum_{\mathbf{x}} \Phi(\mathbf{x}) \psi_{[d_1,
    d_2]}(\mathbf{x})$ then gives that
        \[
            \sum_{f \in V_n(\mathbb{Z})} \phi(f/H) \psi_{[d_1,d_2]}(f)
            =
            H^{n+1} \sum_{\mathbf{u} \in \mathbb{Z}^{n+1}} \widehat \phi
            \left( \frac{\mathbf{u} H}{[d_1,d_2]}\right)
            \widehat \psi_{[d_1,d_2]} (\mathbf{u}).
        \]
    By Corollary~\ref{cor:smooth-mobius}, the right-hand side is equal to
        \[
            \frac{H^{n+1} \widehat\phi(\mathbf{0})}
                 {2^{\omega([d_1,d_2])}}
            +
            O_{\phi,n}([d_1,d_2]^{\frac{3n+5}{4}}).
        \]
    As $d_1$ and $d_2$ are squarefree, one can check that
    $2^{\omega([d_1,d_2])} = \tau(d_1)\tau(d_2)/\tau((d_1,d_2))$, where $\tau$ is the
    divisor function. Substituting this into the full expression from
    Proposition~\ref{prop:modified-selberg}, we obtain
        \begin{equation}\label{eqn:sieve-applied}
            \sum_{\substack{f \in V_n(\mathbb{Z};H)
                            \\ \Gal(f) \subseteq A_n
                            \\ \LDisc(f) \ne 0}}
            \frac{1}{2^{\omega(\LDisc)(f)}}
            \leq
            H^{n+1} \widehat\phi(0) \sum_{d_1,d_2}
            \frac{\lambda_{d_1} \lambda_{d_2}}{\tau(d_1)\tau(d_2)} \tau\big((d_1,d_2)\big)
            +
            O_{\phi,n}\Big(%
              \sum_{d_1,d_2} \lvert \lambda_{d_1} \lambda_{d_2} \rvert [d_1,d_2]^{\frac{3n+5}{4}}
            \Big).
        \end{equation}
    As in the classical Selberg sieve, we diagonalize the quadratic form appearing in the
    first term to obtain
        \begin{align*}
            \sum_{d_1,d_2}
            \frac{\lambda_{d_1} \lambda_{d_2}}{\tau(d_1)\tau(d_2)} \tau\big((d_1,d_2)\big)
            &= \sum_{d_1,d_2} \frac{\lambda_{d_1} \lambda_{d_2}}{\tau(d_1)\tau(d_2)}
               \sum_{e \mid (d_1,d_2)} 1
            \\
            &= \sum_{e} \left(\sum_{d \equiv 0 \Mod{e}} \frac{\lambda_d}{\tau(d)}\right)^2
            \\
            &=: \sum_e \xi_e^2,
        \end{align*}
    say, where the $\xi_e$ are again supported on squarefree integers $e \leq D$. A
    M\"obius inversion argument shows that
        \begin{equation}\label{eq:mobius_inversion}
            \lambda_d = \mu(d)\tau(d) \sum_{e \equiv 0 \Mod{d}} \mu(e) \xi_e.
        \end{equation}
    Thus the constraint that $\lambda_1 = 1$ becomes the condition
        \[
            \sum_{e} \mu(e) \xi_e = 1.
        \]
    This prompts us to choose $\xi_e$ proportional to $\mu(e)$,
        \[
            \xi_e = \frac{\mu(e)}{C}, \quad C := \sum_{e \leq D} \mu(e)^2,
        \]
    so that
        \[
            \sum_e \xi_e^2
                = 1/C
                \ll 1/D.
        \]
    The first term in~\eqref{eqn:sieve-applied} is thus
        \[
            O_{\phi, n} \big(\frac{H^{n+1}}{D}\big).
        \]
    To understand the second term, we note that the choice $\xi_e = \mu(e)/C$
    in~\eqref{eq:mobius_inversion} shows that the terms $\lambda_d$ satisfy
        \[
            \lvert \lambda_d \rvert
            \leq
            \tau(d) \sum_{\substack{e \leq D \\ e \equiv 0 \Mod{d}}} \frac{\mu(e)^2}{C}
            \ll
            \frac{\tau(d)}{d}.
        \]
    Therefore, the second term in~\eqref{eqn:sieve-applied} is
        \[
            \ll_{\phi,n}
            \sum_{d_1,d_2 \leq D} \frac{\tau(d_1)\tau(d_2)[d_1,d_2]^{\frac{3n+5}{4}}}{d_1d_2}
            \ll_{\phi,n}
            \Big(%
              \sum_{d \leq D} \tau(d) d^{\frac{3n+1}{4}}
            \Big)^2
            \ll_{\phi,n}
            D^{\frac{3n+5}{2}} (\log D)^2,
        \]
    where the final bound follows from the (crude) estimate on the Dirichlet
    divisor problem, $\sum_{n \leq X} \tau(n) = O(X \log X)$. Combining these two bounds, we find that
        \[
            \sum_{\substack{f \in V_n(\mathbb{Z};H)
                            \\ \Gal(f) \subseteq A_n
                            \\ \LDisc(f) \ne 0}}
            \frac{1}{2^{\omega(\LDisc)(f)}}
            \ll_{\phi,n}
            \frac{H^{n+1}}{D} + D^{\frac{3n+5}{2}} (\log D)^2.
        \]
    This is optimized by choosing $D = H^{\frac{2n+2}{3n+7}}(\log H)^{\frac{-4}{3n+7}}$,
    which yields the first claim.

    For the monic case, an analogous proof using
    Proposition~\ref{prop:modified-selberg-monic} instead of
    Proposition~\ref{prop:modified-selberg} shows that
        \[
            \sum_{\substack{f \in V_n^{\mathrm{mon}}(\mathbb{Z};H)
                            \\ \Gal(f) \subseteq A_n
                            \\ \Disc(f) \ne 0}}
            \frac{1}{2^{\omega(\Disc(f))}}
            \ll_{\phi,n} \frac{H^n}{D} + D^{\frac{3n+1}{2}} (\log D)^2.
        \]
    Choosing $D = H^{\frac{2n}{3n+3}} (\log H)^{\frac{-4}{3n+3}}$ gives the second claim.
\end{proof}

\section{Almost prime discriminants}\label{sec:almost_prime}

In this section, we apply a weighted almost prime sieve as in~\cite[\S25]{friedlander2010opera}
to obtain lower bounds on almost prime values of polynomial discriminants, in a manner in
spirit with the earlier sections of this paper.
Specifically, we prove

\begin{theorem}\label{thm:almost-prime-polynomial}
    Let $n \geq 3$, and let $H \geq 2$.  For any $r \geq 2n-3$, we have
        \[
            \#\{ f \in V_n^\mathrm{mon}(\mathbb{Z}) : \mathrm{ht}(f) \leq H, \omega(\Disc(f)) \leq r\}
                \gg_{n,r} \frac{H^n}{\log H},
        \]
    where $\omega(\Disc(f))$ denotes the number of distinct primes dividing the
    discriminant of the polynomial $f$.
\end{theorem}

Since the discriminant of a field cut out by an irreducible polynomial divides that of the
polynomial, this also yields lower bounds for the number of degree $n$ number fields with
almost prime discriminant.

\begin{theorem}\label{thm:almost-prime-fields}
    Let $n \geq 3$, and let $X \geq 2$.  For any $r \geq 2n-3$, we have
        \[
            \#\{ F/\mathbb{Q} : [F:\mathbb{Q}]=n, \Disc(F) \leq X, \omega(\Disc(F)) \leq r\}
                \gg_{n,r} \frac{X^{\frac{1}{2}}}{(\log X)^n},
        \]
    where $\omega(\Disc(F))$ denotes the number of distinct primes dividing the
    discriminant of the polynomial $F$.  Moreover, if $c_n \geq 1$ is any constant for which
        \[
            \#\{ F/\mathbb{Q} : [F:\mathbb{Q}]=n, \Gal(\widetilde{F}/\mathbb{Q}) \simeq S_n, \Disc(F) \leq X\}
                \ll_n X^{c_n},
        \]
    then we additionally have
        \[
            \#\{ F/\mathbb{Q} : [F:\mathbb{Q}]=n, \Disc(F) \leq X, \omega(\Disc(F)) \leq r\}
                \gg_{n,r,\epsilon} X^{\frac{1}{2} + \frac{1}{2c_n n (n-1) -2} - \epsilon}.
        \]
\end{theorem}

\begin{remark}
    It is expected that the choice $c_n = 1$ is admissible for every $n$ in
    Theorem~\ref{thm:almost-prime-fields}, but this known only for $n \leq 5$. For $n \geq 6$,
    the smallest known admissible constants are due to Schmidt~\cite{Schmidt} and Lemke Oliver and
    Thorne~\cite{LOTFields}.  It follows from these that the choices $c_n = \frac{n+2}{4}$ and
    $c_n = 1.6 (\log n)^2$ are admissible for every $n \geq 6$, for example.
\end{remark}

In preparation to apply the almost prime sieve, we recall from Lemma~\ref{lem:poly-disc} that given
a monic polynomial $f(x) \in \mathbb{Z}[x]$, a prime $p$ divides the discriminant of $f$ if and only
if $f \Mod{p}$ is not squarefree.

\begin{lemma}\label{lem:squarefree-fourier}
    Let $n \geq 3$ and let $p$ be prime.
    Define $\psi_p\colon V_n^\mathrm{mon}(\mathbb{F}_p) \to \mathbb{C}$ by setting $\psi_p(f) = 1$
    if $f$ is not squarefree and $0$ otherwise.  Then $\widehat\psi_p^\mathrm{mon}(\mathbf{0}) = 1/p$, where $\widehat\psi_p^\mathrm{mon}$ is defined as in \S\ref{FTsec}, and
        \[
            \widehat\psi_p^\mathrm{mon}(\mathbf{v}) \ll p^{-2}
        \]
    for $\mathbf{v} \ne \mathbf{0}$.
\end{lemma}
\begin{proof}
    For $n \geq 2$, the number of monic, squarefree polynomials of degree $n$ over $\mathbb{F}_p$ is
    $p^n - p^{n-1}$.  Thus the number of polynomials that are not squarefree is $p^{n-1}$,
    which yields the claim about $\widehat\psi_p^\mathrm{mon}(\mathbf{0})$, since
        \[
            \widehat\psi_p^\mathrm{mon}(\mathbf{0})
                = \frac{1}{p^n} \sum_{\substack{ f \in V_n^\mathrm{mon} \\ f \text{ not squarefree}}} 1.
        \]
    For
    $\mathbf{v} \ne \mathbf{0}$, we note that $\psi_p(f) = 1 - \mathbf{1}_\mathrm{sf}(f)$, where
    $\mathbf{1}_{\mathrm{sf}}$ is the characteristic function of squarefree polynomials.  Thus for $\mathbf{v} \ne \mathbf{0}$,
    $\widehat\psi_p^{\mathrm{mon}}(\mathbf{v}) = - \widehat{\mathbf{1}}_{\mathrm{sf}}^{\mathrm{mon}}(\mathbf{v})$.
    Mimicking the combinatorial, inclusion-exclusion proof counting squarefree integers, we obtain
        \begin{align*}
            \widehat\psi_p^\mathrm{mon}(\mathbf{v})
                &= \frac{-1}{p^n}
                   \sum_{\substack{ f \in V_n^\mathrm{mon}(\mathbb{F}_p)
                                   \\ f \text{ squarefree}}}
                   e_p( \langle f,\mathbf{v} \rangle_\mathrm{mon})
                \\
                &= \frac{-1}{p^n}
                   \sum_{0 \leq d \leq n/2}
                   \sum_{g \in V_d^\mathrm{mon}(\mathbb{F}_p)} \mu(g)
                   \sum_{f \in V_{n-2d}^\mathrm{mon}(\mathbb{F}_p)}
                   e_p( \langle fg^2,\mathbf{v} \rangle_\mathrm{mon})
                \\
                &= \frac{-1}{p^n}
                   \sum_{0 \leq d \leq n/2}
                   \sum_{g \in V_d^\mathrm{mon}(\mathbb{F}_p)} \mu(g)
                   \sum_{f \in V_{n-2d}^\mathrm{mon}(\mathbb{F}_p)}
                   e_p( \langle f,T_{g^2} \mathbf{v} \rangle_\mathrm{mon}),
        \end{align*}
    where $T_{g^2} \colon \mathbb{F}_p^n \to \mathbb{F}_p^{n-2d}$ is the map adjoint to the (linear)
    map corresponding to multiplication by $g^2$.  The interior sum is a complete sum over
    all polynomials of degree $n-2d$, and hence is $0$ unless $T_{g^2} \mathbf{v} = \mathbf{0}$.
    When $T_{g^2} \mathbf{v} = \mathbf{0}$, the interior summation is $p^{n-2d}$.  Noting also
    that $T_1 \mathbf{v} = \mathbf{v} \ne \mathbf{0}$, it follows that
        \[
            \widehat\psi_p^\mathrm{mon}(\mathbf{v})
            = -\sum_{1 \leq d \leq n/2}
            \sum_{\substack{g \in V_d^\mathrm{mon}(\mathbb{F}_p)
                            \\ T_{g^2} \mathbf{v} = \mathbf{0}}}
            \frac{\mu(g)}{p^{2d}}.
        \]
    We can trivially bound the sum over $d \geq 2$ by ignoring the condition that $T_{g^2}
    \mathbf{v} = \mathbf{0}$,
        \[
            \sum_{2 \leq d \leq n/2} \sum_{g \in V_d^\mathrm{mon}(\mathbb{F}_p)} \frac{1}{p^{2d}}
            \leq
            \sum_{2 \leq d \leq n/2} \frac{1}{p^d} \ll \frac{1}{p^2},
        \]
    which is sufficient.  If $d=1$, then $g = x+\alpha$ for some $\alpha \in \mathbb{F}_p$, and the
    $(n-2)\times n$ matrix $T_{g^2}$ may be written as
        \[
            T_{(x+\alpha)^2}
                = \begin{pmatrix}
                    \alpha^2 & 2\alpha & 1 & 0 & \cdots & 0 \\
                    0 & \alpha^2 & 2\alpha & 1 & \cdots & 0 \\
                    \vdots & 0 & \ddots & \ddots & \ddots & 0 \\
                    0 & \cdots & 0 & \alpha^2 & 2\alpha & 1
                \end{pmatrix}.
        \]
    Since $\mathbf{v} \ne \mathbf{0}$, the equation $T_{g^2} \mathbf{v} = \mathbf{0}$ becomes a system of at most
    quadratic equations in $\alpha$.  This system may or may not have any solutions in $\alpha$, but
    by considering a single non-zero equation, it follows that it admits at most $2$.  Thus the
    contribution from terms with $d=1$ is at most $2p^{-2}$, which is sufficient.
\end{proof}

\begin{remark}\label{largecoefficients}
    Using the argument of Lemma~\ref{lem:squarefree-fourier} but being more careful, it is possible
    to be more precise about the phases $\mathbf{v}$ at which the Fourier transform
    $\lvert \widehat\psi_p^\mathrm{mon}(\mathbf{v})\rvert \gg p^{-2}$, and in general, to identify
    the phases at which the Fourier transform admits worse than the expected
    square-root cancellation.
    We do not presently see a way to exploit this in our proof of
    Theorems~\ref{thm:almost-prime-polynomial}
    and~\ref{thm:almost-prime-fields}, however.
\end{remark}

We are now ready to prove Theorem~\ref{thm:almost-prime-polynomial}. We apply the almost prime sieve
as described by Friedlander and Iwaniec~\cite[Theorem~25.1]{friedlander2010opera}. For convenient
reference, we restate that result here.

\begin{proposition}[Theorem~25.1 of~\cite{friedlander2010opera}]\label{prop:almost_prime_sieve}
    Let $\{ a_n \}$ be a sequence of non-negative numbers which satisfy the linear sieve
    conditions~\cite[(1.2), (5.38)]{friedlander2010opera},
      \begin{equation}\label{eq:linear_sieve_requirements}
          \sum_{\substack{m \leq x \\ m \equiv 0 \Mod d}} a_m
          =
          g(d) X + R_d(x)
          , \qquad \text{and} \qquad
          \prod_{w \leq p < z} \big(1 - g(p)\big)^{-1}
          \leq
          K \Big( \frac{\log z}{\log w} \Big),
      \end{equation}
    for a constant $K > 1$ and any $z > w \geq 2$, where $X$ is to be regarded
    as an approximation to $\sum_{m \leq x} a_n$, the index $p$ runs over
    primes, and $g(d)$ is a multiplicative function.
    Suppose that the remainder terms $R_d(x)$
    satisfy~\cite[(25.7)]{friedlander2010opera}
      \begin{equation}\label{eq:remainder_requirements}
          R(x, D | N)
          :=
          \sum_{d \leq D} \Big\lvert \sum_{n \leq N} \alpha_n R_{dn}(x) \Big\rvert
          \ll X (\log x)^{-3}
      \end{equation}
    for any complex coefficients $\lvert \alpha_n \rvert \leq 1$, and where $D$, $N$
    satisfy~\cite[(25.25)]{friedlander2010opera}
      \begin{equation}\label{eq:DN_requirements}
          D \geq N^{3^r}, \quad DN \geq x^{1/\Delta_r + \epsilon},
      \end{equation}
    in which
      \begin{equation*}
          \Delta_r := r + \frac{1}{\log 3} \log \big( \tfrac{3}{4}(1 + 3^{-r}) \big).
      \end{equation*}

    Let $P(z) := \prod_{p < z} p$ and $V(z) := \prod_{p < z} (1 - g(p))$. Then
    \begin{equation}
      \sum_{\substack{n \leq x \\ (n, P(z)) = 1 \\ \omega(n) \leq r}} a_n
      \asymp
      X V(x)
    \end{equation}
    for $z = (DN)^{\frac{1}{4}}$, and the implied constant depends on $r$ and $\epsilon$.
\end{proposition}

\begin{remark}
  There is a small typo in the statement of this theorem in~\cite{friedlander2010opera}.
  In their theorem statement, $D$ and $N$ need to satisfy (25.25), and not (25.27). Note also
  that they use the notation $\nu(\cdot)$ instead of $\omega(\cdot)$.
\end{remark}

\subsection*{Proof of Theorem~\ref{thm:almost-prime-polynomial}}
    We apply the almost prime sieve as stated in Proposition~\ref{prop:almost_prime_sieve}.

    Let $\phi\colon \mathbb{R}^n \to \mathbb{R}$ be a non-negative Schwartz function
    supported on $[-1,1]^n$. For $H \geq 2$, let $\phi_H(\mathbf{v}) = \phi(\mathbf{v}/H)$. Abusing
    notation, by identifying $V_n^\mathrm{mon}(\mathbb{R})$ with $\mathbb{R}^n$, we may regard
    $\phi$ and $\phi_H$ as Schwartz functions on $V_n^\mathrm{mon}(\mathbb{R})$. For any integer
    $m \geq 1$, let
        \[
            a_m
            :=
            \sum_{\substack{f \in V_n^\mathrm{mon}(\mathbb{Z}) \\ \Disc(f) = \pm m}} \phi_H(f).
        \]
    Since the discriminant of a polynomial in $V_n^\mathrm{mon}(\mathbb{Z})$ of height at most $H$
    is $O_n(H^{2n-2})$, the sequence $a_m$ is supported on integers $m \leq x$ for
    some $x \asymp_n H^{2n-2}$.  Let $d \geq 1$ be a squarefree integer and define
    $\psi_d := \prod_{p \mid d} \psi_p$, where $\psi_p$ is as in Lemma~\ref{lem:squarefree-fourier}.
    Lemma~\ref{lem:poly-disc} implies that $p \mid \Disc(f)$ exactly when $\psi_p(f) = 1$.
    It follows from Poisson summation that
        \begin{align*}
            \sum_{\substack{m \leq x \\ d \mid m}} a_m
                &= \sum_{f \in V_n^\mathrm{mon}(\mathbb{Z})} \phi_H(f) \psi_d(f) \\
                &= H^n \sum_{\mathbf{v} \in \mathbb{Z}^n}
                   \hat\phi\left( \frac{\mathbf{v}H}{d}\right)
                   \widehat\psi_d^\mathrm{mon}(\mathbf{v}) \\
                &= \frac{H^n}{d} \hat\phi(\mathbf{0}) + O_{\phi}(d^{n-2}),
        \end{align*}
    where the last line follows from Lemma~\ref{lem:squarefree-fourier} and
    Lemma~\ref{lem:smooth-fourier-sum}. Recalling Mertens' famous theorem that $\prod_{p \leq x}(1 -
    \frac{1}{p}) = (e^{-\gamma} + o(1))/\log x$, we see that $\{ a_m \}$ satisfies the linear sieve
    conditions~\eqref{eq:linear_sieve_requirements} with $g(d) = 1/d$ and
    $X = H^n \hat\phi(\mathbf{0})$.

    Moreover, the remainders
        \[
            R_d(x) := \sum_{\substack{m \leq x \\ d \mid m}} a_m - \frac{H^n}{d} \hat\phi(\mathbf{0})
        \]
    evidently satisfy
        \[
            R(x, D | 1)
            \leq
            \sum_{d \leq D} \lvert R_d(x) \rvert
            \ll_{\phi}
            D^{n-1}
        \]
    for any $D \geq 1$.  These remainders are $\ll_{\phi, n} X / (\log x)^3$ provided that
    $D \ll_n H^{n/(n-1)} / (\log H)^{3/(n-1)} \asymp x^{n/2(n-1)^2} / (\log x)^{3/(n-1)}$.
    For any such $D$ and $N = 1$, we thus have that $\{ a_m \}$
    satisfies~\eqref{eq:remainder_requirements}.

    The almost prime sieve (Proposition~\ref{prop:almost_prime_sieve}) then shows that for any $r$
    for which~\eqref{eq:DN_requirements} is satisfied, we have the asymptotic
        \[
            \sum_{\substack{m \leq x \\ \omega(m) \leq r}} a_m
            \geq
            \sum_{\substack{m \leq x \\ (m, P(z)) = 1 \\ \omega(m) \leq r}} a_m
            \asymp_{n,\phi} \frac{H^n}{\log H},
        \]
    and~\eqref{eq:DN_requirements} is satisfied when
        \begin{equation}\label{eqn:r-cond}
            \frac{1}{\Delta_r}
            <
            \frac{n}{2(n-1)^2}, \quad \Delta_r
            :=
            r + \frac{1}{\log 3}\log\left(\frac{3}{4}(1+3^{-r})\right).
        \end{equation}
    Noting that $r + \frac{\log(3/4)}{\log 3} < \Delta_r < r$, and
    that $\frac{\log(3/4)}{\log 3} = -0.26\ldots$, the condition~\eqref{eqn:r-cond} is satisfied if
        \[
            r > \frac{2(n-1)^2}{n} + 0.27 = 2n - 3.73 + \frac{2}{n}.
        \]
    In particular, for $n \geq 3$, this is true when $r \geq 2n-3$.  Since
        \[
            \#\{ f \in V_n^\mathrm{mon}(\mathbb{Z}) : \mathrm{ht}(f) \leq H, \omega(\Disc(f)) \leq r\}
                \gg_{\phi} \sum_{\substack{m \leq x \\ \omega(m) \leq r}} a_m,
        \]
    the theorem follows.
\qed{}
\medskip{} 

To go from Theorem \ref{thm:almost-prime-polynomial} to Theorem \ref{thm:almost-prime-fields},
notice that the lower bound in Theorem \ref{thm:almost-prime-polynomial} is larger than the error
term in the Hilbert irreducibility theorem bounds~\eqref{eqn:dietmann-an}
and~\eqref{eq:thm_monic_compare}.
Consequently, the same lower bound holds for the number
of irreducible polynomials with almost prime discriminant, as well as for the number of $S_n$
polynomials with almost prime discriminant.  In particular, almost all of the polynomials produced
by Theorem~\ref{thm:almost-prime-polynomial} cut out $S_n$ fields of degree $n$ with almost prime
discriminant.  To prove Theorem~\ref{thm:almost-prime-fields}, the key is to understand the number
of different polynomials that cut out the same field.  For this, we recall a result of Lemke Oliver
and Thorne~\cite{LOTPolys}.

\begin{lemma}\label{lem:field-multiplicity}
    Let $F$ be a number field of degree $n$, and let
        \[
            M_F(H) = \#\{ f \in \mathbb{Z}[x] : \mathbb{Q}[x]/(f(x)) \simeq F, \mathrm{ht}(f) \leq H\}.
        \]
    Then $M_F(H) \ll_n H (\log H)^{n-1} \Disc(F)^{\frac{-1}{n^2-n}}$, and in particular $M_F(H) \ll_n H (\log H)^{n-1}$.
\end{lemma}
\begin{proof}
    This follows by combining \cite[Theorem 2.1]{LOTPolys} and \cite[Lemma 3.1]{LOTPolys}.
\end{proof}

\subsection*{Proof of Theorem \ref{thm:almost-prime-fields}}

We first prove the statement with the lower bound $\gg_{n,r} X^{1/2} / (\log X)^n$, as it is almost immediate from Theorem \ref{thm:almost-prime-polynomial}, which produces $\gg_{n,r} H^n / \log H$ irreducible polynomials with Galois group $S_n$ whose discriminants have at most $r$ prime factors, and Lemma \ref{lem:field-multiplicity}, which implies that at most $H (\log H)^{n-1}$ of these polynomials can cut out the same field.  In particular, there will be $\gg_{n,r} H^{n-1} / (\log H)^n$ different fields produced, each of which has discriminant $O_n(H^{2n-2})$.  Choosing $H = c X^{1/(2n-2)}$ for a suitable constant $c$ yields the claim.

    To obtain the second claim of the theorem, let $c_n$ be as in the statement of the theorem and suppose $H \geq 2$.
    Then for any $Y \geq 1$, there holds
        \[
            \sum_{\substack{ [F:\mathbb{Q}]=n \\ \Gal(\widetilde{F}/\mathbb{Q}) \simeq S_n \\ \Disc(F) \leq Y}} M_F(H)
                \ll_n H (\log H)^{n-1} \sum_{\substack{ [F:\mathbb{Q}]=n \\ \Gal(\widetilde{F}/\mathbb{Q}) \simeq S_n \\ \Disc(F) \leq Y}}\Disc(F)^{-\frac{1}{n^2-n}}
                \ll_n  H (\log H)^{n-1} Y^{c_n - \frac{1}{n^2-n}}
        \]
    from Lemma \ref{lem:field-multiplicity} and partial summation, where we have used that $c_n \geq 1 > \frac{1}{n^2 - n}$.
    For any $\epsilon > 0$, it follows there is a choice of $Y$ satisfying
        \[
            Y \asymp_{n,\epsilon} H^{\frac{n(n-1)^2}{c_n n (n-1) - 1} - \epsilon}
        \]
    such that
        \[
            \sum_{\substack{ [F:\mathbb{Q}]=n \\ \Disc(F) \leq Y}} M_F(H)
                \ll_{n,\epsilon} H^{n-\epsilon}.
        \]
    This is smaller than the lower bound produced by Theorem \ref{thm:almost-prime-polynomial} on the number of polynomials with almost prime discriminant, almost all of which are irreducible with Galois group $S_n$ by Hilbert irreducibility.  Thus, almost all of the polynomials produced by Theorem \ref{thm:almost-prime-polynomial} cut out degree $n$ $S_n$ extensions $F/\mathbb{Q}$ of discriminant at least $Y$.  For such fields $F$, we have $M_F(H) \ll_n H (\log H)^{n-1} Y^{- \frac{1}{n^2-n}}$ by Lemma \ref{lem:field-multiplicity}.  Dividing the total number of polynomials by this upper bound on the multiplicity, we find
        \begin{align*}
            \#\{F/\mathbb{Q} : [F:\mathbb{Q}] = n,
                &\Gal(\widetilde{F}/\mathbb{Q}) \simeq S_n,
                \omega(\Disc(F)) \leq r, \Disc(F) \ll_n H^{2n-2}\}
            \\
            &\gg_{n,r,\epsilon} H^{n-1} (\log H)^{-n} Y^{\frac{1}{n^2-n}}
            \\
            &\gg_{n,r,\epsilon} H^{n-1 + \frac{n-1}{c_n n(n-1)-1} - \epsilon}.
        \end{align*}
Again choosing $H = c X^{\frac{1}{2n-2}}$ for a suitable constant $c$, the result follows.
\qed{}

\bibliographystyle{alpha}
\bibliography{AIMHIT}

\end{document}